\documentclass[10pt]{amsart}
\usepackage{fancyhdr}
\usepackage[english]{babel}
\usepackage{amssymb}
\usepackage{mathrsfs}
\usepackage{enumerate}
\usepackage{color}
\usepackage{ifpdf}
\usepackage{tikz}
\usepackage{tikz-cd}
\usetikzlibrary{decorations.pathmorphing,arrows}
\usepackage[initials]{amsrefs}
\usepackage{hyperref}

\usepackage{todonotes}

\usepackage{color}

\usepackage[margin=1in]{geometry}  
\usepackage{graphicx}              
\usepackage{amsmath}               
\usepackage{amsfonts}              
\usepackage{amsthm}                
\usepackage{amssymb} 
\usepackage{enumitem}

\theoremstyle{plain}
\newtheorem{thm}{Theorem}[section]
\newtheorem{thmx}{Theorem}

\newtheorem{cor}[thm]{Corollary}
\newtheorem{prop}[thm]{Proposition}

\theoremstyle{definition}
\newtheorem{defn}[thm]{Definition}
\newtheorem{ex}{Example}

\theoremstyle{remark}
\newtheorem*{rem}{\textbf{Remark}}
\newtheorem*{rems}{\textbf{Remarks}}

\newtheoremstyle{case}{}{}{}{}{}{:}{ }{}
\theoremstyle{case}
\newtheorem{case}{Case}

\newcommand{\ME}{\overset{\mathrm{ME}}{\sim}}
\newcommand{\OrE}{\overset{\mathrm{OE}}{\sim}}
\newcommand{\MI}{\overset{\mathrm{MI}}{\rightarrowtail}}
\newcommand{\SMI}{\overset{\mathrm{SMI}}{\rightarrowtail}}
\newcommand{\acts}{\curvearrowright}
\newcommand{\defeq}{\mathrel{\mathop:}=}

\title{Measurable Imbeddings, Free Products, and Graph Products}
\author{\"Ozkan Demir }

\begin{document}

\maketitle

\begin{abstract}
	We study Measurable Imbeddability between two countable groups, which is an order-like generalization of Measure Equivalence that allows the imbedded group to have an infinite measure fundamental domain. We prove if $\Lambda_1$ measurably imbeds into $\Gamma_1$, and $\Lambda_2$ measurably imbeds into $\Gamma_2$ under an additional assumption that lets the corresponding fundamental domains to be arranged in a special way, then $\Lambda_1 * \Lambda_2$ measurably imbeds into $\Gamma_1 * \Gamma_2$. Building upon the techniques used, we show that the analogous result holds for graph products of groups.
\end{abstract}
\section{Introduction and Statement of the Main Results}
Measure equivalence (ME) is an equivalence relation between discrete countable groups that
measure-theoretically generalizes the notion of being virtually isomorphic.
It can be viewed as a measured analogue of quasi-isometry in geometric group theory;
it was introduced by Gromov in \cite{Gro}. The prominent examples of ME groups are lattices in the same locally compact second countable (lcsc) group.

\vspace{1mm}
We study a one-sided generalization of this notion, called Measurable Imbedding (MI), in which one of the groups is allowed to have an infinite measure fundamental domain inside the corresponding coupling. It gives a partial ordering within the class of discrete countable groups. As an example, the definition will cover the following situation: If $\Gamma$ is a lattice in a lcsc group $G$ and $\Lambda$ is any discrete subgroup in $G$, then $\Lambda$ measurably imbeds into $\Gamma$. The purpose of this paper is to show that the notion of meaurable imbeddability is preserved under free products and more general graph products of groups, under a natural assumption on the coupling index.

\vspace{1mm}
A remarkable result involving MI is presented by Gaboriau and Lyons in \cite{GabLy}, where they give a Measurable-Group-Theoretic solution to the famous von Neumann-Day problem. They showed that for any non-amenable group $\Gamma$, one can find an essentially free action of the free group $\mathbb{F}_2 \acts [0,1]^\Gamma$ such that, almost every $\Gamma$-orbit of the Bernoulli action $\Gamma \acts [0,1]^\Gamma$ decomposes into $\mathbb{F}_2$-orbits. Later in \cite{BerVae}, Berendschot and Vaes gave an explicit construction of a Measurable Imbedding (or Measure Equivalence Embedding, as called in their paper) of a free group to any non-amenable group (even in the generality of locally compact groups), in the sense of Definition \ref{def:MI}. This adds yet another item to the endless list of equivalent characterizations of amenability: A group $\Gamma$ is amenable if and only if $\mathbb{F}_2$ does not measurably imbed into $\Gamma$.

\vspace{1mm}
To the author's knowledge, the notion of MI is introduced by Bader, Furman, and Shalom in an unpublished preprint, and first appeared in literature in a paper by Sako (see Definition 2.1 in \cite{Sak}), where he proves that Ozawa's class S is an invariant under Measure Equivalence. He actually proves a stronger statement: if $\Lambda$ measurably imbeds into $\Gamma$ and $\Gamma$ is in class S, then $\Lambda$ is in class S. One can show similar statements hold for amenability and Haagerup's approximantion property (or a-T-menability) by following the same ideas in the proof of them being ME invariants (e.g., see Section 3.1.1 in \cite{Furman}).

\vspace{1mm}
Our main results do not hold at the generality of MI, so we introduce a slight strengthening of this notion called Strict Measurable Imbedding (SMI), which allows us to arrange the fundamental domains so that the fundamental domain having smaller measure is contained in the other. In Proposition \ref{prop:3}, we show that this notion is equivalent to the notion of a Random Subgroup introduced by Monod in \cite{Mon}.

\vspace{1mm}
Orbit Equivalence (OE) is a special case of SMI, which corresponds to the case where we can choose the same set as the fundamental domain for both of the actions. Our first main result is a generalization of Gaboriau's result ($\mathbf{P_{ME}6}$  in \cite{Gab}) which says finite free products of OE groups are OE.

\begin{thmx}
\label{A}
If $\Lambda_1 \SMI \Gamma_1$ and $\Lambda_2 \SMI \Gamma_2$, then $\Lambda_1 * \Lambda_2 \SMI \Gamma_1 * \Gamma_2$
 Moreover, for given $(\Lambda_i\rightarrowtail\Gamma_i)$-couplings $\Sigma_i$ with indices $[\Gamma_i:\Lambda_i]_{\Sigma_i}=c_i$ for $i=1,2$, one can find a $(\Lambda_1 * \Lambda_2\rightarrowtail\Gamma_1 * \Gamma_2)$-coupling $\Sigma$ where:
 \begin{equation*}
    [\Gamma_1 * \Gamma_2:\Lambda_1 * \Lambda_2]_{\Sigma} = 
    \begin{cases}
    1, & \text{if}\ c_1=c_2=1 \\
    \infty, & \text{otherwise}
    \end{cases}
\end{equation*}
\end{thmx}

A generalization of both free products and direct products is the notion of graph products of groups, which is thoroughly studied in Green's thesis (\cite{Green}). In \cite{HH} (Proposition 4.2) Horbez and Huang proved that the graph product of OE groups are OE, generalizing Gaboriau's result. Our second main result generalizes this result and Theorem \ref{A} for finitely generated groups:

\begin{thmx}
\label{B}
Let $\Theta$ be a finite simple irreducible graph with vertex set $V$. Let $H$ and $G$ be two graph products over $\Theta$, with nontrivial finitely generated vertex groups $\{H_v\}_{v \in V}$ and $\{G_v\}_{v \in V }$, respectively. Suppose for each $v \in V$, $H_v \SMI G_v$, with coupling index $c_v$. Then $H \SMI G$ with coupling index $\infty$, unless when all $c_v=1$, in which case the coupling index is also 1.
\end{thmx}

Irreducibility of the underlying graph prevents the decomposition of the graph as the join of two subgraphs, which corresponds to direct product in the group level. The similar result for direct products is known for ME groups, and same proof works for the SMI case. Thus, Theorem \ref{B} holds for non-irreducible graphs, but the index calculation will have some special cases.

\vspace{1mm}
In Section 2 we recall/give definitions and remarks for Measure Equivalence (ME), Measurable Imbeddings (MI), and Strict Measurable Imbeddings (SMI). Then, we investigate the relation between SMI and the notions of randomorphisms, randdembeddings and random subgroups introduced by Monod in \cite{Mon}. In Section 3, we prove Theorem \ref{A}. In Section 4, we give definition of graph products of groups, state the normal form theorem with a relevant proposition and prove Theorem \ref{B}.\\

\textit{Acknowledgements.} I would like to thank my advisor Alexander Furman for his invaluable guidance, support, and encouragement. I would also like to thank the anonymous referee for their helpful and detailed feedback and for pointing out the relation between SMI and randomorphisms.

\section{Definitions, Remarks, and Relation to Randomorphisms}
Throughout the paper; all groups will be discrete countable, all measure spaces will be standard, and all set equalities are almost everywhere (a.e.) equalities (i.e. the symmetric difference between the two sets is null with respect to the ambient measure). We first recall the definition of Measure Equivalence (introduced by Gromov in \cite{Gro}):

\begin{defn}
Two groups $\Lambda$ and $\Gamma$ are called \textbf{Measure Equivalent} (abbreviated as ME, and denoted $\Lambda \ME \Gamma$)  if there exists an infinite measure space $(\Sigma,\mu)$ with a measurable, measure preserving action of $\Lambda \times \Gamma$, so that both actions $\Lambda \acts \Sigma$ and $\Gamma \acts \Sigma$ admit finite-measure fundamental domains $Y \cong \Sigma/\Lambda ,X\cong \Sigma/\Gamma$, i.e.:
\begin{gather*}
    \displaystyle{\Sigma  = \bigsqcup_{\gamma \in \Gamma} \gamma X = \bigsqcup_{\lambda \in \Lambda} \lambda Y}
\end{gather*}

The space $(\Sigma,\mu)$ is called a $(\Lambda,\Gamma)$-\textbf{coupling} or \textbf{ME-coupling}. The \textbf{index} of this coupling is the ratio of the measures of the fundamental domains:
\begin{gather*}
    [\Gamma : \Lambda]_{\Sigma} \defeq \frac{\mu(Y)}{\mu(X)} \in (0,\infty)
\end{gather*}
\end{defn}

\begin{rems}
(See \cite{Furman} for details)
\begin{enumerate}
    \item For a given coupling, the index is well-defined, i.e., it does not depend on the choices of the fundamental domains.
    \item Let $\Lambda \ME \Delta$ and $\Delta \ME \Gamma$ with couplings $\Sigma_1$ and $\Sigma_2$. Then $\Lambda \ME \Gamma$ with the coupling $\Sigma = \Sigma_1 \times_\Delta \Sigma_2 \defeq (\Sigma_1 \times \Sigma_2)/\Delta$ where $\Delta$ acts diagonally on the product. Moreover, the index of this coupling can be computed as:
    $$[\Gamma : \Lambda]_{\Sigma} = [\Gamma : \Delta]_{\Sigma_1} [\Delta : \Lambda]_{\Sigma_2}$$ 
    In particular,  $\ME$ is an equivalence relation on groups.

    \item Any ME-coupling can be decomposed into an integral over a probability space of \textbf{ergodic} ME-couplings, that is, ones for which the $\Lambda \times \Gamma$-action is ergodic.
    \item If $[\Gamma : \Lambda]_{\Sigma} = c \geq 1$, and the coupling is ergodic, then we can choose fundamental domains $X$ and $Y$ (possibly after rescaling the measure $\mu$) so that $X \subset Y$, $\mu(X) = 1$, and $\mu(Y) = c$.
    \item Given a $(\Lambda,\Gamma)$-coupling $\Sigma$ and a choice of fundamental domains $Y \cong \Sigma/\Lambda ,X\cong \Sigma/\Gamma$, we obtain a measurable measure-preserving action $\Lambda \acts X$ where for $\lambda \in \Lambda$ and $x \in X$, $\lambda \cdot x$ is the (essentially) unique element in $\Gamma \lambda x \cap X$, and a cocycle $\alpha : \Lambda \times X \rightarrow \Gamma$ where $\alpha(\lambda,x)$ is the unique element in $\Gamma$ satisfying $\gamma \cdot x = \alpha(\lambda,x)\lambda x$. We similarly have  $\Gamma \acts Y$ and corresponding cocycle $\beta : \Gamma \times Y \rightarrow \Lambda$. (We use $\cdot$ to denote those two actions, to distinguish them from the original action $\Lambda \times \Gamma \acts \Sigma$)
    
\end{enumerate}

\end{rems} 

We give the definition of Measurable Imbedding between groups, which can be thought of as ME but with index in $(0,\infty]$.

\begin{defn}
\label{def:MI}
$\Lambda$ \textbf{Measurably Imbeds} into $\Gamma$ (abbreviated as MI, and denoted $\Lambda \MI \Gamma$) if there exists an infinite measure space $(\Sigma,\mu)$ with a measurable, measure preserving action of $\Lambda \times \Gamma$, so that both actions $\Lambda \acts \Sigma$ and $\Gamma \acts \Sigma$ admit measurable fundamental domains $Y,X \subset \Sigma$ such that $X$ has finite measure.

The space $(\Sigma,\mu)$ is called a $(\Lambda\rightarrowtail\Gamma)$-\textbf{coupling} or \textbf{MI-coupling}. The \textbf{index} of this coupling is the ratio of the measures of the fundamental domains:
\begin{gather*}
    [\Gamma : \Lambda]_{\Sigma} \defeq \frac{\mu(Y)}{\mu(X)} \in (0,\infty]
\end{gather*}
\end{defn}

Here are analogous remarks for MI (the proofs are similar to the ME case):

\begin{rems} $ $
\begin{enumerate}
    \item For a given coupling, the index is well-defined, i.e., it does not depend on the choices of the fundamental domains.
    \item Let $\Lambda \MI \Delta$ and $\Delta \MI \Gamma$ with couplings $\Sigma_1$ and $\Sigma_2$. Then $\Lambda \MI \Gamma$ with the coupling $\Sigma = \Sigma_1 \times_\Delta \Sigma_2$. Same as the ME case, the index of this coupling can be computed as:
    $$[\Gamma : \Lambda]_{\Sigma} = [\Gamma : \Delta]_{\Sigma_1} [\Delta : \Lambda]_{\Sigma_2}$$ 
    In particular,  $\MI$ is an order relation on groups.
    \item Any MI-coupling can be decomposed into an integral over a probability space of \textbf{ergodic} MI-couplings.
    \item We have actions $\Lambda \acts X$ , $\Gamma \acts Y$ and corresponding cocycles $\alpha : \Lambda \times X \rightarrow \Gamma$ , $\beta : \Gamma \times Y \rightarrow \Lambda$ as in the ME case. However, $Y$ might be an infinite measure space in this case, so the cocycle of importance is $\alpha$.
    
\end{enumerate}

\end{rems}

Arranging fundamental domains so that the large group's fundamental domain is contained in the smaller one's plays an important role in our constructions, which is not possible for this general definition if index is less than 1. Indeed, the following example illustrates the theorem does not hold when we replace SMI with MI: consider the trivial coupling of cyclic groups $\mathbb{Z}_2 \MI \mathbb{Z}_2$ with index 1, and another coupling $\mathbb{Z}_3 \MI \mathbb{Z}_2$ with index $\frac{2}{3}$. (As we consider finite groups, these are ME-couplings in disguise.) However, we cannot have $\mathbb{Z}_2 * \mathbb{Z}_3  \MI \mathbb{Z}_2 * \mathbb{Z}_2$, as the latter is amenable while the former is not.

\vspace{1mm}
Even if the coupling index is greater than 1, it is not always possible to arrange fundamental domains as described, as the following example demonstrates:

\begin{ex}
\label{ex:1}
Let $(\Sigma_i,\mu_i)$ be a $(\Lambda\rightarrowtail\Gamma)$-couplings with indices $c_i$ so that $c_1<1$, $c_2>1$. Possibly after rescaling the measures $\mu_i$, we can choose fundamental domains $Y_i \cong \Sigma_i/\Lambda ,X_i\cong \Sigma_i/\Gamma$ for $i=1,2$, so that $\mu_1(Y_1) = c_1$, $\mu_1(X_1) = 1$, $\mu_2(Y_2) = ac_2$, $\mu_2(X_2) = a$, where $a$ is chosen so that the ratio $ \displaystyle{c =\frac{ac_2 + c_1}{a+1}>1}$. Then the coupling $\Sigma = \Sigma_1 \bigsqcup \Sigma_2$ has index $c>1$, but it is not possible to arrange fundamental domains so that the one with the larger measure contains the other. The main reason is the fact that the coupling $\Sigma$ is not ergodic, as $\Sigma_1, \Sigma_2 \in \Sigma$ are invariant sets which are neither null nor conull.
\end{ex}

To exclude such cases, we introduce the following strengthening of MI:

\begin{defn}
$\Lambda$ \textbf{Strictly Measurably Imbeds} into $\Gamma$ (abbreviated as SMI, and denoted $\Lambda \SMI \Gamma$) if there exists an infinite measure space $(\Sigma,\mu)$ with a measurable, measure preserving action of $\Lambda \times \Gamma$, so that both actions $\Lambda \acts \Sigma$ and $\Gamma \acts \Sigma$ admit Borel fundamental domains $Y,X \subset \Sigma$ such that $X$ has finite measure and $X \subset Y$.

The space $(\Sigma,\mu)$ is called a $(\Lambda\rightarrowtail\Gamma)$-\textbf{coupling} or \textbf{SMI-coupling}. The \textbf{index} of this coupling is the ratio of the measures of the fundamental domains:
\begin{gather*}
    [\Gamma : \Lambda]_{\Sigma} \defeq \frac{\mu(Y)}{\mu(X)} \in [1,\infty]
\end{gather*}
\end{defn}

We similarly have well-definedness of the coupling index, transitivity, and cocycles for SMI. Recall that in Example \ref{ex:1}, what prevented the MI-coupling from being an SMI-coupling was the fact that $\Sigma_1,\Sigma_2 \subset \Sigma$ were nontrivial invariant subsets, so the coupling was not ergodic. Indeed in general, if $\Sigma$ is an ergodic MI-coupling with index $c\geq1$, then $\Sigma$ is an SMI-coupling with the same index. 
On the other hand, if there is an MI-coupling with index $c\geq1$, then by ergodic decomposition we can find an ergodic MI-coupling with a possibly different index $\tilde{c}\geq 1$, hence an SMI-coupling. These observation give us the following characterization of SMI in terms of MI:
\begin{prop}
\label{prop:0}
$\Lambda \SMI \Gamma$ if and only if $\Lambda \MI \Gamma$ with index $c\geq1$.
\end{prop}

Next, we have the following proposition that gives us the necessary and sufficient properties an SMI-cocycle must have.

\begin{prop}
\label{prop:1}
$\Lambda \SMI \Gamma$ if and only if there exist a probability space $(X, \mu)$  with a measure preserving $\Lambda$-action, and a cocycle $\alpha: \Lambda \times X \rightarrow \Gamma$ satisfying $\alpha(\lambda,x) = e \iff \lambda = e$, for all $x \in X$.
\end{prop}

\begin{proof}
If we have $\Lambda \SMI \Gamma$, choosing the fundamental domains $X \subset Y$ as in the definition gives the action $\Lambda \acts X$, and the corresponding cocycle $\alpha$. Using the fact that $X\subset Y$, one can show that this cocycle satisfies the property in the statement.

For the other direction, suppose we have such a probability space with a $\Lambda$-action and a $\Lambda$-valued cocycle. Let $(\Omega,\tilde{\mu}) \defeq (\Gamma \times X, c_{\Lambda} \times \mu)$, where $c_{\Lambda}$ is the counting measure. Define the actions of $\Gamma$ and $\Lambda$ on $\Omega$ as:
 \begin{align*} 
     \Gamma \acts \Omega   \qquad \; \gamma \bullet (\bar{\gamma},x) &= (\gamma\,\bar{\gamma},x) \\
     \Lambda \acts \Omega  \qquad \; \lambda \bullet (\bar{\gamma},x) &= (\bar{\gamma}\,\alpha(\lambda,x)^{-1},\lambda \cdot x)
 \end{align*}
These actions are measure preserving, they commute, and $\{e\} \times X$ is a $\Gamma$-fundamental domain. One can check that both actions are free, the induced $\Lambda$-action on X is the action $\cdot$, and the cocycle we get is $\alpha$. Now, we will describe a $\Lambda$-fundamental domain.
Enumerate the group $\Gamma$ as $\Gamma = \{\gamma_0, \gamma_1, \gamma_2, ... \}$, where $\gamma_0 = e$. Let $Y_0 \defeq \{e\} \times X$. Having constructed $Y_0, Y_1, ... , Y_{n-1}$ inductively, define:
$$Y_n = (\{\lambda_n\} \times X) \setminus \Lambda \bullet \left( \bigsqcup_{i=0}^{n-1} Y_i \right)$$

Finally, set $\displaystyle{Y \defeq \bigsqcup_{i = 0}^\infty Y_i}$. Using the construction and the hypothesis on the cocycle $\alpha$, one can check that $Y$ is a $\Lambda$-fundamental domain. It contains the $\Gamma$-fundamental domain $\{e\} \times X$. Hence, $\Omega$ is the desired $(\Lambda\rightarrowtail\Gamma)$-coupling.
\end{proof}

\begin{rem}
    The condition on the cocycle $\alpha$ is equivalent to the following statement: For all $x \in X$, $\alpha(e,x) = e$ and the functions $\alpha(\, \cdot \, ,x):\Lambda \rightarrow \Gamma$ are injective. So, SMI-cocycles are characterized by this injectivity property.
\end{rem}

Remaining of the section follows \cite{Mon} for the definitions.
Given two countable groups  $\Lambda$ and $\Gamma$, one can define the Polish space
$$ \textbf{[}\Lambda , \Gamma\textbf{]} \defeq \{f:\Lambda\rightarrow\Gamma \,| \,f(e)=e\} $$
There is a natural $\Lambda$-action on $\textbf{[}\Lambda , \Gamma\textbf{]}$ given by
$$ (\lambda\cdot f)(x) \defeq f(x\lambda)f(\lambda)^{-1} \text{  for  } f \in \textbf{[}\Lambda , \Gamma\textbf{]}, \; \lambda,x \in \Lambda.$$

\begin{defn} (\textbf{Definitions 5.1 \& 5.2} in \cite{Mon}) A \textbf{randomorphism} from $\Lambda$ to $\Gamma$ is a $\Lambda$-invariant probability measure on $ \textbf{[}\Lambda , \Gamma\textbf{]}$. A randomorphism is a \textbf{randembedding} if it is supported on the injective maps. Say that $\Lambda$ is a \textbf{random subgroup} of $\Gamma$ if it admits a randembedding into $\Gamma$.
\end{defn}

We observe that being a random subgroup is equivalent to SMI.

\begin{prop}
\label{prop:3}
    $\Lambda \SMI \Gamma$ if and only if $\Lambda$ is a random subgroup of $\Gamma$.
\end{prop}
\begin{proof}
The forward direction follows from the discussion at Section 5.2 in \cite{Mon}. Indeed, arranging the fundamental domains as $X \subset Y$, the SMI-cocycle $\alpha : \Lambda \times X \rightarrow \Gamma$ yields a measurable map $\hat{\alpha} : X \rightarrow \textbf{[}\Lambda , \Gamma\textbf{]}$ given by $\hat{\alpha}(x)(\lambda) = \alpha(\lambda,x)$. This map is $\Lambda$-equivariant, and one can check the condition $X \subset Y$ ensures that $\hat{\alpha}$ lands into injective maps. Therefore, letting $\mu$ be the probability measure on X, the measure $\hat{\alpha}_*\mu$ is a randembedding. 

\vspace{1mm}

For the other direction, let $X \defeq \{f:\Lambda\rightarrow\Gamma \,| \,f(e)=e, \, f \text{ is injective}\}$ equipped with the randembedding $\mu$, which is invariant under the action $\Lambda \acts X$, $(\lambda\cdot f)(x) \defeq f(x\lambda)f(\lambda)^{-1}$. Consider the space \hbox{$\Omega \defeq \{f:\Lambda\rightarrow\Gamma \,| \,f \text{ is injective}\}$} with the $\Gamma$-action by translations on the image, i.e. $(\gamma f)(x) \defeq \gamma (f(x))$. Note $\Omega = \bigsqcup_{\gamma \in \Gamma} \gamma X$, so equipping $\Omega$ with the natural measure coming from this partition makes the $\Gamma$-action measure preserving, and X becomes a fundamental domain for $\Gamma$. Let $\Lambda \acts \Omega$ by $(\lambda f)(x) \defeq f(\lambda)f(x\lambda)f(\lambda)^{-1}$. $\Omega$ consisting of injective functions implies that this $\Lambda$-action is free, and this action commutes with the $\Gamma$ action. Moreover, the induced $\Lambda$-action on X coming from these free commuting actions is the original $\Lambda$-action on X. Hence, we will get a cocycle $\alpha : \Lambda \times X \rightarrow \Gamma$ such that $\gamma \cdot f = \alpha(\lambda,f)\lambda f$. This formula shows that $\alpha(\lambda,f) = f(\lambda)$. By the definition of the space $X$, this cocycle satisfies the hypotheses of Proposition \ref{prop:1}, which gives the result.

 \end{proof}


\section{Free Products}

A special case of ME is Orbit Equivalence (OE), which corresponds to having a coupling with index 1. If additionally the actions are ergodic, one can choose the same set as the fundamental domain for both of the actions. (See \cite{Furman} for the definition and proof of this correspondence.) Gaboriau's result mentioned in the introduction is:

\begin{thm} ($\mathbf{P_{ME}6}$  in \cite{Gab})
If $\Lambda_1 \OrE \Gamma_1$ and $\Lambda_2 \OrE \Gamma_2$, then $\Lambda_1 * \Lambda_2 \OrE \Gamma_1 * \Gamma_2$
\end{thm}

 There, he actually proves the same result for infinite free products, from which the above result follows. He also gives counterexamples to show that the conclusion does not hold if we replace OE with ME. Theorem \ref{A} shows that the result is true if we replace OE with SMI. Theorem \ref{A} follows from the following theorem using transitivity of SMI:
 
 \begin{thm}
 \label{thm:1}
 For groups $\Lambda$, $\Gamma$, $G$; if $\Lambda \SMI \Gamma$ with coupling $\Sigma$ and $G$ is non-trivial, then $G * \Lambda \SMI G * \Gamma$ with coupling index $\infty$, unless $[\Gamma:\Lambda]_{\Sigma} = 1$, in which case the new coupling index is also 1.
 \end{thm}
 
 \begin{proof}
 
 Fix a $(\Lambda,\Gamma)$-coupling $(\Sigma,\mu)$, fundamental domains $Y \cong \Sigma/\Lambda ,X\cong \Sigma/\Gamma$ such that $X \subset Y$ and $\mu(X) = 1$. Then we have the action $\Lambda \acts X$ and the corresponding cocycle $\alpha: \Lambda \times X \rightarrow \Gamma$ so that $\lambda \cdot x = \alpha(\lambda,x)\lambda x$. \hfill
 
 Let $G$ act on $X$ trivially to get an action of $G*\Lambda$ on $X$, still denoted by $\cdot$. Extend the cocycle $\alpha$ to $\tilde{\alpha}: (G*\Lambda)\times X \rightarrow G*\Gamma$ by setting $\tilde{\alpha}(g,x) = g$, for $g \in G$, and using the cocycle identity \linebreak $\tilde{\alpha}(gh,x) = \tilde{\alpha}(g,h \cdot x) \tilde{\alpha}(h,x)$. \hfill
 
 Set $\widetilde{\Sigma} = (G*\Gamma)\times X$, with the induced product measure $\tilde{\mu}$, where we take the counting measure on $G*\Gamma$, and $\mu|_X$ on $X$.  We define the following actions:
 \begin{align*} 
     G*\Gamma \acts \widetilde{\Sigma}   \qquad \; w \cdot (\bar{w},x) &= (w\,\bar{w},x) \\
     G*\Lambda \acts \widetilde{\Sigma}  \qquad \; w \cdot (\bar{w},x) &= (\bar{w}\,\tilde{\alpha}(w,x)^{-1},w \cdot x)
 \end{align*}

This is the same construction from the proof of Proposition \ref{prop:1}, and it is easy to check that the cocycle $\tilde{\alpha}$ satisfies the hypothesis of that Proposition. So we can apply Proposition \ref{prop:1} and already conclude that $\widetilde{\Sigma}$ is a $(G*\Lambda_i\rightarrowtail G*\Gamma_i)$-coupling. In order to calculate the coupling index, we will give an explicit description of a $G*\Lambda$-fundamental domain.
 
 We will look at $\widetilde{\Sigma}$ from another perspective, which will make the description of this fundamental domain easier. The idea is to see $\widetilde{\Sigma}$ as a collection of disjoint copies of the original coupling $\Sigma$, using the natural identification $\Sigma \cong \Gamma \times X$. Letting $\Gamma \acts G*\Gamma$ by right multiplication, we have $G*\Gamma \cong G*\Gamma/\Gamma \times \Gamma$. By seeing elements of $G*\Gamma$ as reduced words, and choosing the representatives of minimal word length from $G*\Gamma/\Gamma$ gives the identification $G*\Gamma/\Gamma \cong W$, where:
$$ W \defeq \{w \in G*\Gamma | \, w  \text{ is a reduced word ending with a non-trivial element of } G \} \cup \{e\} $$ \hfill

This gives us the identification $\widetilde{\Sigma} \cong W \times \Gamma \times X$. The action of $G*\Lambda$ making this identification $G*\Lambda$-equivariant is:
\begin{align*}
    g \bullet (w, \gamma, x) &= (w\gamma g^{-1}, e, x) \quad &\text{for}\, g \in G \\
    \lambda \bullet (w, \gamma, x) &= (w, \gamma \alpha(\lambda,x)^{-1}, \lambda \cdot x) \quad &\text{for}\, \lambda \in \Lambda
\end{align*}

If $z\in G*\Lambda$ is a reduced word starting with a nontrivial element of $G$, a straightforward calculation shows:
$$z \bullet (w, \gamma, x) = (w\gamma \tilde{\alpha}(z,x)^{-1},e,z\cdot x)$$

Identifying $\Gamma \times X \cong \Sigma$, we claim $\widetilde{Y} \defeq ( \{e\} \times X) \bigsqcup (W \times (Y\setminus X))$ is a fundamental domain for $G*\Lambda$. We first show that $G*\Lambda \bullet \widetilde{Y} = W \times \Gamma \times X$, so $\widetilde{Y}$ intersects every orbit of $G*\Lambda$. Indeed, the identification $W \times \Gamma \times X \cong W \times \Sigma$ is $G*\Gamma$-equivariant, and $W \times X \subset (G*\Lambda)\bullet(\{e\} \times X)$. This gives $W \times Y \subset (G*\Lambda)\bullet\widetilde{Y}$. As $Y$ is a $\Lambda$ fundamental domain in $\Sigma$, $\Lambda \bullet (\{w\} \times Y) = \{w\} \times \Sigma$. Hence, $W \times \Gamma \times X \cong W \times \Sigma \subset \Lambda \bullet (W \times Y) \subset (G*\Lambda) \bullet \widetilde{Y}$.

\vspace{2mm}

It remains to show for any nontrivial element $w \in G*\Lambda$ and $y \in \widetilde{Y}$, $w \cdot y \notin \widetilde{Y}$.
An arbitrary element of $G*\Lambda$ can be written of the form $\lambda z$, where $\lambda \in \Lambda$ and $z \in G*\Lambda$ is a reduced word starting with an element of G. Also note that $(w,\gamma,x) \in \widetilde{Y}$ if and only if $w=\gamma=e$ corresponding to $\{e\} \times X$, or $\gamma x \in Y$ as an element in $\Sigma$. We also recall the following observation coming from the choice of the fundamental domains so that $X \subset Y$:
\begin{equation}
    \alpha(\lambda,x) = e \quad \iff \quad \lambda = e
    \label{eq:1}
\end{equation} \hfill

Letting $p:G*\Gamma \rightarrow \Gamma$ be the retraction homomorphism onto $\Gamma$, note that for $w \in G*\Gamma$ and $x \in X$, $w \cdot x = p(w) \cdot x$, as $G$ acts trivially on $X$. So, we have the following equality:
\begin{equation}
    w \cdot x = \alpha(p(w),x)p(w)x
    \label{eq:2}
\end{equation}

We have three cases depending on whether $\lambda$ or $z$ are trivial:

\begin{case}
Both $\lambda$ and $z$ are non-trivial.

Let $(w,\gamma,x) \in \widetilde{Y}$. Then,
    \begin{equation*}
        \lambda z \bullet (w,\gamma,x) = \lambda \bullet (w \gamma \tilde{\alpha}(z,x)^{-1}, e, z \cdot x)= (w \gamma \tilde{\alpha}(z,x)^{-1}, \alpha(\lambda, z \cdot x)^{-1}, \lambda z \cdot x) \notin \widetilde{Y}
    \end{equation*}
    as  $\tilde{\alpha}(\lambda, z \cdot x) \neq e$ by (\ref{eq:1}) and:
    \begin{align*}
        \alpha(\lambda, z \cdot x)^{-1} (\lambda z \cdot x) &= \alpha(\lambda, p(z) \cdot x)^{-1} \alpha(\lambda p(z),x)\lambda p(z) x \quad & \text{by }(\ref{eq:2}) \\
        &= \alpha(p(z),x)\lambda p(z)x & \text{by cocycle identity} \\
        &= \lambda (p(z) \cdot x) & \text{by }(\ref{eq:2}) 
    \end{align*}

    $\lambda (p(z) \cdot x) \notin \widetilde{Y}$ as $p(z) \cdot x \in X \subset Y$ and $\lambda \neq e$. 
\end{case}

\begin{case}
$\lambda = e$ and $z$ is non-trivial.
    \begin{equation*}
        z \bullet (e,e,x) = (\tilde{\alpha}(z,x)^{-1}, e, z \cdot x) \notin \widetilde{Y}
    \end{equation*}
    as $\alpha(z,x) \neq e$ by (\ref{eq:1}) and definition of $\tilde{\alpha}$. \hfill
    
    For $\gamma \neq e$ and $\gamma x \in Y$:
   \begin{equation*}
        z \bullet (w,\gamma,x) = (w \gamma \tilde{\alpha}(z,x)^{-1}, e, z \cdot x) \notin \widetilde{Y}
    \end{equation*}
    as $w \gamma \tilde{\alpha}(z,x)^{-1} \neq e$. Indeed, $\gamma x \in Y$ implies $\alpha(\lambda,x) \neq \gamma$ for all $\lambda \in \Lambda$ and for a.e. $x \in X$, so $\gamma$ inside the word $w \gamma \tilde{\alpha}(z,x)^{-1}$ will not be cancelled.
\end{case}

\begin{case}
$z = e$ and $\lambda$ is non-trivial.
    \begin{equation*}
        \lambda \bullet (e,e,x) = (e, \alpha(\lambda, x)^{-1}, \lambda \cdot x) \notin \widetilde{Y}
    \end{equation*}
    as $\alpha(\lambda,x) \neq e$ by (\ref{eq:1}), and 
    \begin{equation*}
        \alpha(\lambda, x)^{-1} (\lambda \cdot x) = \lambda x \notin Y
    \end{equation*}
        
    For $\gamma \neq e$ and $\gamma x \in Y$:
   \begin{equation*}
        \lambda \bullet (w,\gamma,x) = (w, \gamma \alpha(\lambda,x)^{-1}, \lambda \cdot x) \notin \widetilde{Y}
    \end{equation*}
    as $\gamma \alpha(\lambda,x)^{-1} \neq e$ as in Case 2, and
    \begin{equation*}
        \gamma \alpha(\lambda,x)^{-1} (\lambda \cdot x) = \gamma \lambda x = \lambda (\gamma x) \notin Y
    \end{equation*}
    as $\gamma x \in Y$ and $\lambda \neq e$.
\end{case}

This shows $\widetilde{Y}$ is a $G * \Lambda$ fundamental domain. Note if $[\Gamma:\Lambda]_{\Sigma} = 1$, then $Y=X$, so $\tilde{\mu}(\widetilde{Y}) = \tilde{\mu}(\{e\} \times X) = \mu(X) = 1$. Otherwise, $[\Gamma:\Lambda]_{\Sigma} > 1$, so $\mu(Y\setminus X) > 0$. Hence, we have
$\tilde{\mu}(\widetilde{Y}) = \tilde{\mu}(\{e\} \times X) + \tilde{\mu}(W \times (Y\setminus X)) = 1 + |W|\mu(Y \setminus X) = \infty$, as $W$ is infinite as long as $G$ is a non-trivial group. This proves the statement about the coupling indices.
\end{proof}

\begin{proof}[Proof of Theorem \ref{A}]
By Theorem \ref{thm:1}, we get $\Lambda_1 * \Lambda_2 \SMI \Gamma_1 * \Lambda_2$ with index $d_1$, and $\Gamma_1 * \Lambda_2 \SMI \Gamma_1 * \Gamma_2$ with index $d_2$. Composing these couplings gives $\Lambda_1 * \Lambda_2 \SMI \Gamma_1 * \Gamma_2$ with index $d_1 d_2$. If $c_1 = c_2 = 1$, then $d_1 = d_2 =  1$, so $d_1d_2 = 1$. Otherwise, at least one $d_i$ is $\infty$, so $d_1 d_2 = \infty$.
\end{proof}

The following is immediate using Proposition \ref{prop:3}.

\begin{cor}
If $\Lambda_i$ is a random subgroup of $\Gamma_i$ for $i=1,2$, then $\Lambda_1 * \Lambda_2$ is a random subgroup of $\Gamma_1*\Gamma_2$.
\end{cor}

\section{Graph Products}
\setcounter{case}{0}

We first state some preliminary definitions for graph products, and a normal form theorem for the elements in a graph product. Throughout this section, the groups are assumed to be finitely generated, and the graphs are finite and simple, i.e. with no edge loops and with no multiple edges between vertices. For proofs and further details, see \cite{Green}.

\begin{defn} (Definition 3.1 in \cite{Green})
Let $\Theta$ be a graph, with vertex set $V\Theta = \{ v_i | i = 1,2, ... , n \}$. Let $\{A_k\}_{k=1}^n$ be a collection of groups. Let $G\Theta \defeq \langle A_k \, | \; [\, A_i,A_j]\, \: \forall (\,v_i, v_j)\, \in E\Theta \rangle$, where \hbox{$E\Theta$ is the edge set of $\Theta$.} We call $G\Theta$ the \textbf{graph product} of the $\{A_k\}$ given by $\Theta$. We call $\Theta$ the \textbf{underlying graph}, and the groups $A_k$ are called the \textbf{generating vertex groups}.
\end{defn}

If $\Theta$ is a complete graph, then the graph product reduces to the direct product of the vertex groups. On the other hand, if there are no edges in $\Theta$, the graph product becomes the free product of the vertex groups. So graph products give us a spectrum of constructions between these two extremities. We have the notion of a word being reduced in graph products analogous to the same notion in free products where no cancellation can happen between the elements.

\begin{defn} (Definition 3.5 in \cite{Green})
\label{def:1}
Let $G\Theta$ be a graph product of the groups $A_1, A_2, ..., A_n$. A sequence $g_1, g_2, ... g_m$ of elements from $G\Theta$ is called $\textbf{reduced}$ if:
\begin{enumerate}[label=(\roman*)]
\item $g_i \in A_i$, for $i = 1,2,...,n$ .
\item $g_i \neq e$, for $i = 1,2,...,n$ .
\item $\forall g_i,g_j$ with $i<j$ for which $\exists k \; i \leq k < j$ with
\begin{align*}
    [\,g_i,g_{i+1}]\, &= [\,g_i,g_{i+2}]\, = ... = [\,g_i,g_k]\, = e \\
    [\,g_{k+1},g_j]\, &= [\,g_{k+2},g_j]\, = ... = [\,g_{j-1},g_j]\, = e
\end{align*}
 $g_i$ and $g_j$ are in different generating groups.
\end{enumerate}
\end{defn}

Part (iii) means that repeatedly swapping elements from generating groups whose corresponding vertices are adjacent in $\Theta$ cannot bring together two terms from the same generating group. We will refer to this swapping of commuting terms as \textbf{syllable shuffling}.

\begin{defn} (Definition 3.7 in \cite{Green})
Let $G\Theta$ be a graph product of the groups $A_1, A_2, ..., A_n$. Fix finitely many generators for each $A_i$, and let $w$ be a word in those generators.
If $w = w_1 w_2 ... w_r$  where each $w_j$ is a word in the generators of only one of the generating groups, no $w_j$ is the empty word, and $w_j$, $w_{j+1}$ are not in the same generating group for $j = 1,2, ...,r-1$, then the \textbf{syllable length}, $\lambda(w)$ of $w$ is $r$, and $w_1, w_2, ..., w_r$ are called the \textbf{syllables} of $w$. The \textbf{syllable length} $\lambda(g)$ of an element $g \in G\Theta$ is the minimal syllable length of a word defining $g$.
\end{defn}

The following theorem, gives a normal form for elements in graph products, which plays the role of reduced words in free products.

\begin{thm} (Theorem 3.9 in \cite{Green})
\label{thm:nf}
Let $G\Theta$ be a graph product of the groups $A_1, A_2, ..., A_n$. Each element $g \neq e$ of $G\Theta$ can be uniquely (up to syllable shuffling) expressed as a product
$$
g = g_1 g_2 ... g_r \quad \text{where} \; g_1, g_2, ..., g_r \; \text{is a reduced sequence.}
$$
\end{thm}
\begin{proof}
See \cite{Green}.
\end{proof}

The \textbf{link of a vertex} v, denoted $\textbf{lk(v)}$, is the set of all vertices that are connected to v. The \textbf{star of a vertex} v, denoted $\textbf{st(v)}$, is $lk(v) \cup \{v\}$.

\begin{prop}
\label{prop:2}
Let $G\Theta$ be a graph product of the groups $A_1, A_2, ..., A_n$. Fix a vertex generator $A_k$, and let $L$ be the standard subgroup of $G\Theta$ generated by $lk(v_k)$, where $v_k$ is the vertex corresponding to $A_k$. Then, each element $g \in G\Theta$ can be expressed as a product $g = alh$ where:
\begin{enumerate}[label=(\roman*)]
    \item $a \in A_k$
    \item $l$ is a reduced word in $L$.
    \item $h$ is a reduced word starting with an element of $A_m$, where the corresponding vertex is not in $st(v_k)$, or $h=e$
\end{enumerate}
\end{prop}

\begin{proof}
For $g=e$, we just have $a=l=g=e$. \hfill

For $g \neq e$, write $g = g_1 g_2 ... g_r $ as in Theorem \ref{thm:nf}. Note that, either it is possible to have $g_1 \in A_k$ after syllable shuffling, in which case $a=g_1$; or not, in which case $a=e$. Then, let $g_k$ be the first element not in $A_k$ or $H$. In case $a=g_1$, taking $l = g_2 g_3 ... g_{k-1}$ and $h = g_k g_{k+1} ... g_r$ gives the result. (If $k=2$, $l=e$) In case $a=e$, taking $l = g_1 g_2 g_3 ... g_{k-1}$ and $h = g_k g_{k+1} ... g_r$ gives the result. (If $k=1$, $l=e$) If such a $g_k$ does not exist, just take $l = a^{-1}g$ and $h=e$. These give the result in all possible cases.
\end{proof}

It is straightforward to show that if $\Lambda_i \SMI \Gamma_i$ with index $c_i$ for i=1,2; then $\Lambda_1 \times \Lambda_2 \SMI \Gamma_1 \times \Gamma_2$, and there is a $(\Lambda_1 \times \Lambda_2 \rightarrowtail \Gamma_1 \times \Gamma_2)$-coupling with index $c_1c_2$, with the convention $\infty \cdot c = \infty$. (Use the direct product of $(\Lambda_i \rightarrowtail \Gamma_i)$-couplings with the natural action of $\Lambda_1 \times \Lambda_2$ and $\Gamma_1 \times \Gamma_2$.) Indeed the same result holds if we replace SMI with MI. As the direct product case is covered using these observations, we will not deal with the graph products which split into direct products, so we need the following notion of irreducibility for graphs:

\begin{defn} 
Given two graphs $\Theta_1$ and $\Theta_2$, define their \textbf{join}, denoted $\Theta_1 \circ \Theta_2$, as the graph we get by connecting every vertex of $\Theta_1$ to every vertex of $\Theta_2$ with an edge. More rigorously: 
\begin{align*}
    V(\Theta_1 \circ \Theta_2) &= V\Theta_1 \cup V\Theta_2 \\
    E(\Theta_1 \circ \Theta_2) &= E\Theta_1 \cup E\Theta_2 \cup \{(v_1,v_2) | v_1 \in V\Theta_1, v_2 \in V\Theta_2\}
\end{align*}
A graph $\Theta$ is called \textbf{irreducible}, if there are no pair of nonempty full subgraphs $\Theta_1$ and $\Theta_2$, such that $\Theta = \Theta_1 \circ \Theta_2$.

\end{defn}

Theorem \ref{B} follows from the following result by using transitivity of SMI:

\begin{thm}
\label{thm:2}
Let $\Theta$ be a finite simple irreducible graph with vertex set $V$. Let $H$ and $G$ be two graph products over $\Theta$, with nontrivial finitely generated vertex groups $\{H_v\}_{v \in V}$ and $\{G_v\}_{v \in V }$, respectively. Suppose there exists $w \in V$ suvh that, $H_w \SMI G_w$ with coupling index $c$, and $H_v$ and $G_v$ are isomorphic for all $v \in V, \; v\neq w$. Then $H \SMI G$ with coupling index $\infty$, unless when $c=1$, in which case the coupling index is also 1.
\end{thm}

\begin{proof}
Set $\Lambda = H_w$, $\Gamma = G_w$. Fix a $(\Lambda,\Gamma)$-coupling $(\Sigma,\mu)$, fundamental domains $Y \cong \Sigma/\Lambda$ and $X\cong \Sigma/\Gamma$, such that $X \subset Y$ and $\mu(X) = 1$. Then we have the action $\Lambda \acts X$ and the corresponding cocycle $\alpha: \Lambda \times X \rightarrow \Gamma$ so that $\lambda \cdot x = \alpha(\lambda,x)\lambda x$.
\hfill

For each $v\neq w$, let $H_v$ act on $X$ trivially. That gives an action of the graph product $H$ on $X$, still denoted by $\cdot$. Extend the cocycle $\alpha$ to $\tilde{\alpha}: H\times X \rightarrow G$ by setting $\tilde{\alpha}(h,x) = h$, for $h \in H_v$, $v\neq w$ and using the cocycle identity $\tilde{\alpha}(gh,x) = \tilde{\alpha}(g,h \cdot x) \tilde{\alpha}(h,x)$  
\hfill

Set $\widetilde{\Sigma} = (G\times X, \tilde{\mu})$, where $\tilde{\mu}$ is the induced product measure.  We define the following actions:
 \begin{align*}
     G \acts \widetilde{\Sigma}   \quad  w \cdot (\bar{w},x) &= (w\,\bar{w},x) \\
     H \acts \widetilde{\Sigma}  \quad  w \cdot (\bar{w},x) &= (\bar{w}\,\tilde{\alpha}(w,x)^{-1},w \cdot x)
 \end{align*}
 
 As before, this is the same construction from the proof of Proposition \ref{prop:1}, so the same remarks hold. We can use Proposition \ref{prop:2} to show that the cocycle $\tilde{\alpha}$ satisfies the hypothesis of Proposition \ref{prop:1}, so we can already conclude that $H \SMI G$. In order to calculate the coupling index, we will give an explicit description of an $H$-fundamental domain.
 \hfill

 Letting $\Gamma \acts G$ by right multiplication, we have $G \cong G/\Gamma \times \Gamma$. By seeing elements of $G$ as reduced words (as in Definition \ref{def:1}) up to syllable shuffling, and choosing representatives of minimal word length from $G/\Gamma$ gives the identification $G/\Gamma \cong W$, where
$$ W \defeq \{w \in G | \, w  \text{ is not equal to a nontrivial reduced word ending with an element of } \Gamma \} \cup \{e\} $$
\hfill

This gives us the identification $\widetilde{\Sigma} \cong W \times \Gamma \times X$. The action of $H$ making this identification $H$-equivariant is defined by:
\begin{align*}
    \lambda \bullet (w, \gamma, x) &= (w, \gamma \alpha(\lambda,x)^{-1}, \lambda \cdot x) \quad &\text{for}&\, \lambda \in \Lambda = H_v \\
     h \bullet (w, \gamma, x) &= (w h^{-1}, \gamma, x) \quad &\text{for}&\, h \in H_w \, \text{, where}\, w\in lk(v) \\
    h \bullet (w, \gamma, x) &= (w\gamma h^{-1}, e, x) \quad &\text{for}&\, h \in H_w \,\text{, where}\, w\notin st(v)
\end{align*}

Let $H_{lk(v)} < H$ be the standard subgroup of $H$ generated by $lk(v)$. Note $H_{lk(v)} \acts W$ by right multiplication. Let $\widetilde{W} \subset W$ be a fundamental domain for this action, for example, choose the minimal length word from each $H_{lk(v)}$-orbit. Note that if $V\Theta = st(v)$, then $\Gamma$ would be in the center of $G$, and we would have $W=H_{lk(v)}$, and  $\widetilde{W} = \{e\}$. However, that would mean $\Theta = \{v\} \circ lk(v)$, contradicting irreducibility. Hence, $W$ properly contains $H_{lk(v)}$, and in this case it is easy to observe that $\widetilde{W}$ is nontrivial. This is the only place we use irreducibility of $\Theta$.

Identifying $W \times \Gamma \times X \cong W \times \Sigma$ equivariantly, we claim $\widetilde{Y} \defeq ( \{e\} \times X) \bigsqcup (\widetilde{W} \times (Y\setminus X))$ is a fundamental domain for G. 

\vspace{1mm}
Showing $G \bullet \widetilde{Y} = W \times \Gamma \times X$ is similar to the free product case after we observe \hbox{$W \times \Sigma \subset H_{lk(v)} \bullet (\widetilde{W} \times \Sigma)$.} Hence, $\widetilde{Y}$ intersects every $G$-orbit.

\vspace{1mm}
It remains to show that for any nontrivial element $w \in H$ and $y \in \widetilde{Y}$, $w \cdot y \notin \widetilde{Y}$. Using Proposition \ref{prop:2}, such an element $w$ can be written of the form $w = \lambda gh$ where $\lambda \in \Lambda$, $g$ is a reduced word in $H_{lk(v)}$, and $h \in H$ is a reduced word starting with an element not in $H_{lk(v)}$ or $\Lambda$. Also note that $(w,\gamma,x) \in \widetilde{Y}$ if and only if $w=\gamma=e$ corresponding to $\{e\} \times X$ or $w\in \widetilde{W}$ and $\gamma x \in Y$ as an element in $\Sigma$. As in the proof of Theorem \ref{thm:1}, we have (\ref{eq:1}), and letting $p:H \rightarrow \Gamma$ be the retraction, we have (\ref{eq:2}). For such h and g, we can compute the action as:
\begin{align*}
    h \bullet (w, \gamma, x) &= (w\gamma \tilde{\alpha}(h,x)^{-1},e,h\cdot x) \\
    g \bullet (w, \gamma, x) &= (w g^{-1},\gamma,x)
\end{align*}

We have seven cases, depending on whether $\lambda$, $g$, or $h$ are trivial:

\begin{case}
$\lambda$, $g$, $h$ are non-trivial.
\\
Let $(w,\gamma,x) \in \widetilde{Y}$. Then,
\begin{equation*}
        \lambda g h \bullet (w,\gamma,x) = (w \gamma \tilde{\alpha}(h,x)^{-1} g^{-1}, \alpha(\lambda, h \cdot x)^{-1}, \lambda h \cdot x) \notin \widetilde{Y}
    \end{equation*}
    as  $\alpha(\lambda,h \cdot x) \neq e$ by (\ref{eq:1}), and
    \begin{align*}
        \alpha(\lambda, h \cdot x)^{-1} (\lambda h \cdot x) &= \alpha(\lambda, p(h) \cdot x)^{-1} \alpha(\lambda p(h),x)\lambda p(h) x \quad & \text{by }(\ref{eq:2}) \\
        &= \alpha(p(h),x)\lambda p(h)x & \text{by cocycle identity} \\
        &= \lambda (p(h) \cdot x) & \text{by }(\ref{eq:2}) 
    \end{align*}
    $\lambda (p(h) \cdot x) \notin \widetilde{Y}$ as $p(h) \cdot x \in X \subset Y$ and $\lambda \neq e$. 
\end{case}

\begin{case}
$\lambda = e$, $g$ and $h$ are non-trivial.
    \begin{equation*}
        g h \bullet (e,e,x) = (\tilde{\alpha}(h,x)^{-1} g^{-1}, e, h \cdot x) \notin \widetilde{Y}
    \end{equation*}
    as $\alpha(h,x)^{-1} g^{-1} \neq e$ 
    \hfill
    
    For $w \in \widetilde{W}$, $\gamma \neq e$ and $\gamma x \in Y$:
   \begin{equation*}
        g h \bullet (w,\gamma,x) = (w \gamma \tilde{\alpha}(h,x)^{-1} g^{-1}, e, h \cdot x) \notin \widetilde{Y}
    \end{equation*}
    as $w \gamma \tilde{\alpha}(h,x)^{-1} g^{-1} \neq e$. Indeed, $\gamma x \in Y$ implies $\alpha(\lambda,x) \neq \gamma$ for all $\lambda \in \Lambda$ and for a.e. $x \in X$, so $\gamma$ inside the word will not be cancelled.
\end{case}

\begin{case}
$g = e$, $\lambda$ and $h$ are non-trivial.
\hfill

Same as Case 1 from the proof of Theorem \ref{thm:1}.
\end{case}

\begin{case}
$h = e$, $\lambda$ and $g$ are non-trivial.
    \begin{equation*}
        \lambda g \bullet (e,e,x) = (g^{-1}, \alpha(\lambda, x)^{-1}, \lambda  \cdot x) \notin \widetilde{Y}
    \end{equation*}
    as $g \neq e$, and $\alpha(\lambda, x)^{-1} (\lambda  \cdot x) = \lambda x \notin Y$.
    \hfill
    
    For $w \in \widetilde{W}$, $\gamma \neq e$ and $\gamma x \in Y$:
\begin{equation*}
        \lambda g \bullet (w,\gamma,x) = (w g^{-1}, \gamma \alpha(\lambda, x)^{-1}, \lambda  \cdot x) \notin \widetilde{Y}
    \end{equation*}
    as $\gamma \alpha(\lambda, x)^{-1} \neq e$ and $\gamma \alpha(\lambda, x)^{-1} (\lambda  \cdot x) = \lambda \gamma x \notin Y$.
\end{case}

\begin{case}
$\lambda = g = e$, and $h$ is non-trivial.
\hfill

Same as Case 2 from the proof of Theorem \ref{thm:1}.
\end{case}

\begin{case}
$h = g = e$, and $\lambda$ is non-trivial.
\hfill

Same as Case 3 from the proof of Theorem \ref{thm:1}.
\end{case}

\begin{case}
$\lambda = h = e$, and $g$ is non-trivial.
\begin{equation*}
        g \bullet (e,e,x) = (g^{-1}, e, x) \notin \widetilde{Y}
    \end{equation*}
    as $g \neq e$.
\hfill

For $w \in \widetilde{W}$, $\gamma \neq e$ and $\gamma x \in Y$:
    \begin{equation*}
        g \bullet (w,\gamma,x) = (w g^{-1}, \gamma, x) \notin \widetilde{Y}
    \end{equation*}
    as $w g^{-1} \notin \widetilde{W}$ by the definition of $\widetilde{W}$.
\end{case}

The calculation of the index of this coupling is same as the one in the proof of Theorem \ref{thm:1}. 
\end{proof}

\begin{proof}[Proof of Theorem \ref{B}]
The final coupling we get at the end will be the composition of all the intermediate couplings described in Theorem \ref{thm:2}, with index equal to the product of the intermediate indices. If $c_v = 1$ for all $v$, then all of the intermediate couplings will also have index 1, so product will be 1. Otherwise, at least one of them will have infinite index, so the product will be $\infty$.
\end{proof}

\begin{rem}
The same result holds without the irreducibility assumption up to the index calculation. Without irreducibility, we might encounter cases where $\widetilde{W} = {e}$, so the coupling index can be $\tilde{\mu}(\widetilde{Y}) = \mu(Y)$. That means for such intermediate couplings, the index will be $c_v$. In particular, this gives another proof for the case of direct products.
\end{rem}

We finish by combining this remark with Proposition \ref{prop:3}:

\begin{cor}
Let $\Theta$ be a finite simple graph with vertex set $V$. Let $H$ and $G$ be two graph products over $\Theta$, with nontrivial finitely generated vertex groups $\{H_v\}_{v \in V}$ and $\{G_v\}_{v \in V }$, respectively. Suppose for each $v \in V$, $H_v$ is a random subgroup of $G_v$. Then $H$ is a random subgroup of $G$.
\end{cor}


\end{document}